\documentclass[a4paper,twoside,11pt,reqno]{amsart}

\usepackage{amsmath,amsfonts,amsbsy,amsgen,amscd,mathrsfs,amssymb,amsthm,bm}
\usepackage{enumerate,mathtools}

\usepackage[usenames,dvipsnames]{xcolor}
\usepackage[colorlinks=true,citecolor=blue,linkcolor=blue]{hyperref}
\usepackage{comment}

\usepackage{tikz}
\usetikzlibrary{shadings}

\newtheorem{theorem}{Theorem}
\newtheorem{lemma}[theorem]{Lemma}

\newtheorem*{prop*}{Proposition}
\newtheorem{conj}{Conjecture}
\newtheorem*{conj*}{Conjecture}

\newtheorem*{fact*}{Fact}

\newtheorem*{ex*}{Example}

\newcommand*{\claimproofname}{Proof of claim}

\newcommand{\eps}{\varepsilon}

\newcommand{\bC}{\mathbf{C}}

\DeclareMathSymbol{\shortminus}{\mathbin}{AMSa}{"39}

\theoremstyle{definition}

\numberwithin{equation}{section} 
\numberwithin{figure}{section}
\numberwithin{table}{section}

\newcommand{\bE}{\mathbf{E}}

\newcommand{\bF}{\mathbf{F}}

\newcommand{\1}{(-1)}

\allowdisplaybreaks

\begin{document}
	
\title[Fourier growth]{Fourier growth of degree $2$ polynomials}

 \begin{abstract}
    We prove bounds for the absolute sum of all level-$k$ Fourier coefficients for $\1^{p(x)}$, where polynomial $p:\bF_2^n \to \bF_2$ is of degree $1$ or degree $2$.
\end{abstract}

\author[L. Becker]{Lars Becker}
\address{(L.B.) Mathematics Department, Princeton University, Fine Hall, Washington Road, Princeton, NJ 08544-1000, USA}
\email{lbecker@math.princeton.edu}

\author[J. Slote]{Joseph Slote}
\address{(J.S.) Department of Computing and Mathematical Sciences, California Institute of Technology, Pasadena, CA}
\email{jslote@uw.edu}

\author[A. Volberg]{Alexander Volberg}
\address{(A.V.) Department of Mathematics, MSU, 
East Lansing, MI 48823, USA, and Hausdorff Center of Mathematics, Bonn, Germany}
\email{volberg@math.msu.edu}

\author[H. Zhang]{Haonan Zhang}
\address{(H.Z.) Department of Mathematics, University of South Carolina}
\email{haonanzhangmath@gmail.com}

%\subjclass[2010]{46B10, 46B09; 46B07; 60E15}

\thanks{ This work was completed while the  authors were in residence at the
Hausdorff Research Institute for Mathematics in Bonn, during the trimester
program ``Boolean Analysis in Computer Science", funded by the Deutsche
Forschungsgemeinschaft (DFG, German Research Foundation) under Germany's
Excellence Strategy – EXC-2047/1 – 390685813; }
\thanks{L.B. is supported by the Collaborative Research Center 1060 funded by the Deutsche Forschungsgemeinschaft;}

\thanks{J.S. is supported by Chris Umans’ Simons Institute Investigator Grant;}
\thanks{A.V. is supported by NSF  DMS-2154402 and by the Hausdorff Center for Mathematics;}
\thanks{H.Z. is supported by NSF  DMS-2453408.}

\subjclass[2010]{68Q17; 46B10; 60E15}

\keywords{Pseudorandom generators,  Fourier weight of $\bF_2$ polynomials,  vectors of a fixed boolean weight}
	
\maketitle

\section{Introduction}
	
Let $\bF_2=\{0,1\}.$ Any function 
$f: \bF_2^n \to \bC$ has the Fourier--Walsh expansion
\begin{equation}
    f(x)=\sum_{A\subset [n]}\widehat{f}(A)(-1)^{x_A},
\end{equation}
where $A\subset [n]:=\{1,2,\dots, n\}$ and $x_A = \sum_{i \in A} x_i$. The degree of $f$ is defined as $\deg(f):=\max_{\widehat{f}(A)\neq 0}|A|$, where $|A|$ denotes the cardinality of $A$. 

Boolean functions $f:\bF_2^n\to \{-1,1\}$ play an important role in theoretical computer science and social choice theory. Of particular interest to us are those of the form $f(x)=(-1)^{p(x)}$ with $p:\bF^n_2\to \bF_2$ a polynomial of small degree.  

We want to estimate independently of $n$ the $k$-th Fourier weight of $\1^{p(x)}$, where $p:\bF_2^n\to \bF_2$ a polynomial of degree at most $d$. Here, the $k$-th Fourier weight of a function $f: \bF_2^n \to \bC$ refers to the quantity
\[
    \sum_{|A| = k} |\widehat{f}(A)|\,,
\]
where 
\[
    \widehat f(A) = \frac{1}{2^n} \sum_{x \in \bF_2^n} f(x) (-1)^{x_A}\,.
\]
In fact, Chattopadhyay, Hatami, Hosseini and Lovett \cite[Open problem 1.7]{CHHL} conjectured the following. 

\begin{conj}
    Let $f(x)=(-1)^{p(x)},x\in \bF_2^n$ with $p:\bF_2^n\to \bF_2$ being a polynomial of degree at most $d$. Then there exists a constant $c_d>0$ depending only $d$ such that 
    \begin{equation*}
        \sum_{A:|A|=k}|\widehat{f}(A)|\le c_d^k
    \end{equation*}
    for all $k\ge 1$. In particular, the constant $c_d=2^{O(d)}$ should work. 
\end{conj}

As explained in \cite{CHHL}, it is necessary to assume the exponential dependence on $k$. In case $\deg(p)\le 1$, we may write $p(x)=a+x_{i_1}+\cdots +x_{i_m}$ for $1\le i_1<\dots<i_m\le n$ and some $a\in \bF_2$. Then 
\begin{equation*}
    |\sum_{x\in \bF_2^n}(-1)^{p(x)+x_A}|
    =\begin{cases}
    2^n & A=\{i_1,\dots, i_m\}\\
    0. & \textnormal{otherwise}
    \end{cases}
\end{equation*}
Therefore, we have
\begin{equation*}
    \sum_{A\subset [n]}|\widehat{f}(A)|
    =\frac{1}{2^n}\sum_{A\subset [n]}|\sum_{x\in \bF_2^n}(-1)^{p(x)+x_A}|
    =1
\end{equation*}
and thus 
\begin{equation*}
    \sum_{A:|A|=k}|\widehat{f}(A)|\le 1.
\end{equation*}
%the absolute sum of all Fourier coefficients of $(-1)^{p(x)}$ is exactly $1$. In particular the $k$-th Fourier weight is at most $1$ for all $k$. 

For $p$ of degree at most 2, we will prove by a simple induction argument the following.
\begin{theorem}
    \label{thm deg 2}
    For any polynomial $p: \bF^n_2 \to \bF_2$ of degree at most $2$, the $k$-th Fourier weight of $f(x)=(-1)^{p(x)}$ is at most $(1 + \sqrt{2})^k$,
    that is, 
    \begin{equation*}
        \sum_{A:|A|=k}|\widehat{f}(A)|\le (1+\sqrt{2})^k.
    \end{equation*}
\end{theorem}
This confirms a conjecture in \cite{CHHL} for polynomials of degree $2$. We note that the authors state in \cite{CHHL} to already have a proof in that case, which appears however not to be published. In Section \ref{dickson}, we give another proof of Theorem \ref{thm deg 2} based on the strategy they had in mind. We thank Pooya Hatami for kindly sharing it with us. It gives the following sharp upper bound.

\begin{theorem}
    \label{thm deg 2 sharp}
    For any polynomial $p: \bF^n_2 \to \bF_2$ of degree at most $2$, the $k$-th Fourier weight of $f(x)=(-1)^{p(x)}$ is
    \begin{equation*}
        \sum_{A:|A|=k}|\widehat{f}(A)|\le \sqrt{\frac{2}{\pi}} k^{-1/2} (1+\sqrt{2})^k.
    \end{equation*}
    This is sharp in the sense that there exist pairs $(p,k)$ for which the ratio of right hand side and left hand side gets arbitrarily close to $1$.
\end{theorem}

\medskip

Pseudorandom generators and fractional pseudorandom generators are constructed in \cite{CHHL}, Section 4, 
given that the exponential bounds on Fourier weight of  $(-1)^p$ for high degree polynomials  $p$ are proved.

Pseudorandom generators  are widely studied in complexity theory. There are several general
frameworks used to approach their possible construction. The literature is vast, one can read a survey \cite{G} for example. 
A harmonic analysis approach was suggested in 
\cite{CHHL}, and it requires the good estimate of Fourier weights. The reader can see other approaches in e. g. \cite{NW}, \cite{HLW}.

In \cite{CGLLS} it was  shown that for
polynomial error, Fourier bounds on the first $O(\log n)$ levels is sufficient to recover the seed
length (the smallness of pseudorandom generator) obtained in \cite{CHLT}, which required bounds on the entire tail.

\section{Proof of Theorem \ref{thm deg 2} by induction}
\label{firstproof}

We denote by $W(2,k,n)$ the maximum over all polynomials of degree $2$ on $\bF_2^n$ of the $k$-th Fourier weight of $(-1)^{p(x)}$:
\begin{equation}
\label{W def}
    W(2,k,n) = \max_{\mathrm{deg} \, p \le 2} \sum_{A \subset [n], |A| = k} \left|\frac{1}{2^n} \sum_{x \in \bF_2^n} (-1)^{p(x)}  (-1)^{x_A}\right|\,.
\end{equation}
Here we abbreviate $x_A = \sum_{i \in A} x_i$. 
%We further denote the maximum $k$-th Fourier weight among all $n$ variable degree $2$ polynomials by $W(2,k,n)$:
% $$
%     W(2,k,n) = \max_{\mathrm{deg} \, p \le 2} \sum_{A \subset [n], |A| = k} \left|\frac{1}{2^n} \sum_{x \in \bF_2^n} (-1)^{p(x)}  (-1)^{x_A}\right|\,.
% $$
%In Theorem 6.1 of \cite{CHHL} is is shown that
% $$
%     W(2,k) \le (64 k )^k\,.
% $$
% Following the argument in the previous section, we can improve this to a bound for $W(2,k)$ that is exponential in $k$.
Theorem \ref{thm deg 2} claims that for all $0 \le k \le n$
\begin{equation}
    \label{eq weight}
        W(2,k, n) \le (1 + \sqrt{2})^{k}\,.
\end{equation}
We will prove this by induction on the number of variables $n$ and on $k$. Note that the estimate is obvious when $k = 0$ or $k = n$, since then we have $W(2,k,n) \le 1$. So let $k \ge 1$ and $n \ge k+1$.

For our induction step, we assume that for all $k' < k$ and all $n'$
\begin{equation}
    \label{induction1}
    W(2, k',n') \le (1 + \sqrt{2})^{k'} 
\end{equation}
and for all $n' < n$
\begin{equation}
    \label{induction2}
    W(2,k,n') \le (1 + \sqrt{2})^k\,.
\end{equation}
Our goal is to deduce from this that 
$$
    W(2,k,n) \le (1 + \sqrt{2})^k\,,
$$
this will complete the proof of the theorem.

Let $p$ be a maximizing polynomial of degree at most $2$ in $n$ variables, such that
\begin{align}
    W(2,k, n) &= \sum_{A \subset [n], |A| = k} \left|\frac{1}{2^n} \sum_{x \in \bF_2^n} (-1)^{p(x)}  (-1)^{x_A}\right|\label{eps def}.
\end{align}

If $p$ is affine then we have 
$$
    p(x) = a + x_{i_1} + \dotsb + x_{i_m}
$$
for some $m$ and $a \in \bF_2$. In that case, as explained above, \eqref{eps def} is $1$ if $k = m$ and $0$ if $k \neq m$, so $W(2,k, n) \le 1$.

If $p$ is not affine, then we write it, perhaps after permuting the variables, as
\begin{equation}\label{p(x)}
	    p(x) = p_0(\bar x) + x_1 p_1(\bar x) + x_2 p_2(\bar x) + x_1 x_2\,,
\end{equation}
where $\bar x = (x_3, \dotsc, x_n)$. For a set $A \subset [n]$, we abbreviate $\bar A = A \setminus [2]=A \cap \{3, \dotsc, n\} $. According to our notation, $\bar{x}_{\bar{A}}=\sum_{i\in [n]\setminus [2]}x_i$. Writing out the sum over $x_1, x_2$ in \eqref{eps def}, we then obtain
\begin{align*}
	    W(2,k, n)
	    =I_{0}+I_{1}+I_{2}+I_{12}
\end{align*}
where (we omit the constraints $A\subset [n]$ and $|A| = k$ for notational convenience)
\begin{align*}
	I_{0}&=\sum_{A \cap\{1,2\}=\emptyset} \left|\frac{1}{2^n} \sum_{x \in \bF_2^n} (-1)^{p(x)}  (-1)^{x_A}\right|,\\
		I_{1}&=\sum_{A \cap\{1,2\}=\{1\}} \left|\frac{1}{2^n} \sum_{x \in \bF_2^n} (-1)^{p(x)}  (-1)^{x_A}\right|,\\
			I_{2}&=\sum_{A \cap\{1,2\}=\{2\}} \left|\frac{1}{2^n} \sum_{x \in \bF_2^n} (-1)^{p(x)}  (-1)^{x_A}\right|,\\
			I_{12}&=\sum_{A \cap\{1,2\}=\{1,2\}} \left|\frac{1}{2^n} \sum_{x \in \bF_2^n} (-1)^{p(x)}  (-1)^{x_A}\right|.
\end{align*}
Recall that by \eqref{p(x)}
$$p(0,0,\bar{x})=p_0(\bar{x}),\qquad p(0,1,\bar{x})=p_0(\bar{x})+p_2(\bar{x})$$
and 
$$p(1,0,\bar{x})=p_0(\bar{x})+p_1(\bar{x}),\qquad p(1,1,\bar{x})=1+p_0(\bar{x})+p_1(\bar{x})+p_2(\bar{x}).$$
 Consider the term $I_0$ first. Since $A \cap\{1,2\}=\emptyset$, we have $(-1)^{x_A}=(-1)^{\bar{x}_{\bar{A}}},$
and thus 
\begin{equation*}
	I_0=\frac{1}{2^n}\sum_{A \cap\{1,2\}=\emptyset} \left|\sum_{\bar{x} \in \bF_2^{n-2}} (-1)^{\bar{x}_{\bar{A}}+p_0(\bar{x})}  \left(1+(-1)^{p_1(\bar{x})}+(-1)^{p_2(\bar{x})}-(-1)^{p_1(\bar{x})+p_2(\bar{x})} \right)\right|.
\end{equation*}
The other three terms can be rewritten in a similar way.  For $I_1$, since $A \cap\{1,2\}=\{1\}$, we have $(-1)^{x_A}=(-1)^{x_1+\bar{x}_{\bar{A}}},$
and thus 
\begin{equation*}
	I_1=\frac{1}{2^n}\sum_{A \cap\{1,2\}=\{1\}} \left|\sum_{\bar{x} \in \bF_2^{n-2}} (-1)^{\bar{x}_{\bar{A}}+p_0(\bar{x})}  \left(1-(-1)^{p_1(\bar{x})}+(-1)^{p_2(\bar{x})}+(-1)^{p_1(\bar{x})+p_2(\bar{x})} \right)\right|.
\end{equation*}
For $I_2$, since $A \cap\{1,2\}=\{2\}$, we have $(-1)^{x_A}=(-1)^{x_2+\bar{x}_{\bar{A}}},$
and thus 
\begin{equation*}
	I_2=\frac{1}{2^n}\sum_{A \cap\{1,2\}=\{2\}} \left|\sum_{\bar{x} \in \bF_2^{n-2}} (-1)^{\bar{x}_{\bar{A}}+p_0(\bar{x})}  \left(1+(-1)^{p_1(\bar{x})}-(-1)^{p_2(\bar{x})}+(-1)^{p_1(\bar{x})+p_2(\bar{x})} \right)\right|.
\end{equation*}
For $I_{12}$, since $A \cap\{1,2\}=\{1,2\}$, we have $(-1)^{x_A}=(-1)^{x_1+x_2+\bar{x}_{\bar{A}}},$
and thus 
\begin{equation*}
	I_{12}=\frac{1}{2^n}\sum_{A \cap\{1,2\}=\{1,2\}} \left|\sum_{\bar{x} \in \bF_2^{n-2}} (-1)^{\bar{x}_{\bar{A}}+p_0(\bar{x})}  \left(1-(-1)^{p_1(\bar{x})}-(-1)^{p_2(\bar{x})}-(-1)^{p_1(\bar{x})+p_2(\bar{x})} \right)\right|.
\end{equation*}

%\[
%    W(2,k, n)
%\]
%\[
%    = \frac14 \frac1{2^{n-2}} \sum _{\bar x \in \bF_2^{n-2}} \1^{p_0(\bar x)}\cdot
%\]
%\[
%    \Big[ \sum_{A\cap \{1,2\} = \emptyset} \varepsilon_A\1^{x_{\bar A}} + \sum_{A\cap \{1,2\} = \{1\}}\varepsilon_A\1^{x_{\bar A}} 
%    + \sum_{A\cap \{1,2\} = \{2\}} \varepsilon_A\1^{x_{\bar A}}+  \sum_{\{1,2\} \subset A} \varepsilon_A\1^{x_{\bar A}}\Big] 
%\]
%\[
%    +\frac14 \frac1{2^{n-2}} \sum _{\bar x \in \bF_2^{n-2}} \1^{p_0(\bar x)}\1^{p_1(\bar x)}\cdot
%\]
%\[
%    \Big[ \sum_{A\cap \{1,2\} = \emptyset} \varepsilon_A\1^{x_{\bar A}} - \sum_{A\cap \{1,2\} = \{1\}}\varepsilon_A\1^{x_{\bar A}} + \sum_{A\cap \{1,2\} = \{2\}} \varepsilon_A\1^{x_{\bar A}}-  \sum_{\{1,2\} \subset A} \varepsilon_A\1^{x_{\bar A}}\Big] 
%\]
%\[
%    + \frac14 \frac1{2^{n-2}} \sum _{\bar x \in \bF_2^{n-2}} \1^{p_0(\bar x)}\1^{p_2(\bar x)}\cdot
%\]
%\[
%    \Big[ \sum_{A\cap \{1,2\} = \emptyset} \varepsilon_A\1^{x_{\bar A}} + \sum_{A\cap \{1,2\} = \{1\}}\varepsilon_A\1^{x_{\bar A}} - \sum_{A\cap \{1,2\} = \{2\}} \varepsilon_A\1^{x_{\bar A}}-  \sum_{\{1,2\} \subset A} \varepsilon_A\1^{x_{\bar A}}\Big] 
%\]
%\[
%    -\frac14 \frac1{2^{n-2}} \sum _{\bar x \in \bF_2^{n-2}} \1^{p_0(\bar x)} \1^{p_1(\bar x)} \1^{p_2(\bar x)}\cdot
%\]
%\[
%    \Big[ \sum_{A\cap \{1,2\} = \emptyset} \varepsilon_A\1^{x_{\bar A}} - \sum_{A\cap \{1,2\} = \{1\}}\varepsilon_A\1^{x_{\bar A}} - \sum_{A\cap \{1,2\} = \{2\}} \varepsilon_A\1^{x_{\bar A}}+  \sum_{\{1,2\} \subset A} \varepsilon_A\1^{x_{\bar A}}\Big].
%\]
We will simplify these formulae using the following Lemma. 

\begin{lemma}
    For $a,b\in \bF_2$, we have 
    \begin{equation}
         1 + (-1)^{a} + (-1)^{b} - (-1)^{a+b} = 2(-1)^{ab}\,.
    \end{equation}
\end{lemma}

\begin{proof}
    The proof follows easily by direct computation. We provide here, however, a complicated proof. Note that $(-1)^x=1-2x,x\in \bF_2$. So 
    $$
    (-1)^{a+b}=(-1)^{a}(-1)^{b}=(1-2a)(1-2b)=1-2(a+b)+4ab.
    $$
    So 
    \begin{align*}
          1+(-1)^{a} +(-1)^{b}-2(-1)^{ab}
  &=1+1-2a+1-2b-2(1-2ab)\\
  &=1-2(a+b)+4ab\\
  &=(-1)^{a+b}. \qedhere
    \end{align*}
\end{proof}

So for polynomials $p_1, p_2$ taking values in $\bF_2$, we have
\[
    1 + (-1)^{p_1} + (-1)^{p_2} - (-1)^{p_1 + p_2} = 2(-1)^{p_1 p_2}\,.
\]
Replacing $p_1$ by $p_1 + 1$ and $p_2$ by $p_2 + 1$ yields also 
\[
    1 - (-1)^{p_1} + (-1)^{p_2} + (-1)^{p_1 + p_2} = 2(-1)^{p_1 p_2 + p_2}\,,
\]
\[
    1 + (-1)^{p_1} - (-1)^{p_2} + (-1)^{p_1 + p_2} = 2(-1)^{p_1 p_2 + p_1}\,,
\]
and
\[
    1 - (-1)^{p_1} - (-1)^{p_2} - (-1)^{p_1 + p_2} = 2(-1)^{p_1 p_2 + p_1 + p_2}\,.
\]
Using these four identities, we may rewrite the above four terms as 
$$
I_0=\frac{1}{2}\sum_{\bar{A}\subset [n]\setminus[2],|\bar{A}|=k} \left|\frac{1}{2^{n-2}}\sum_{\bar{x} \in \bF_2^{n-2}} (-1)^{p_0(\bar{x})+p_1(\bar{x})p_2(\bar{x})}(-1)^{\bar{x}_{\bar{A}}} \right|,
$$
$$
I_1=\frac{1}{2^{n-1}}\sum_{\bar{A}\subset [n]\setminus[2],|\bar{A}|=k-1} \left|\sum_{\bar{x} \in \bF_2^{n-2}} (-1)^{p_0(\bar{x})+p_2(\bar{x})+p_1(\bar{x})p_2(\bar{x})}(-1)^{\bar{x}_{\bar{A}}} \right|,
$$
$$
I_2=\frac{1}{2^{n-1}}\sum_{\bar{A}\subset [n]\setminus[2],|\bar{A}|=k-1} \left|\sum_{\bar{x} \in \bF_2^{n-2}} (-1)^{p_0(\bar{x})+p_1(\bar{x})+p_1(\bar{x})p_2(\bar{x})}(-1)^{\bar{x}_{\bar{A}}} \right|,
$$
and 
$$
I_{12}=\frac{1}{2^{n-1}}\sum_{\bar{A}\subset [n]\setminus[2],|\bar{A}|=k-2} \left|\sum_{\bar{x} \in \bF_2^{n-2}} (-1)^{p_0(\bar{x})+p_1(\bar{x})+p_2(\bar{x})+p_1(\bar{x})p_2(\bar{x})}(-1)^{\bar{x}_{\bar{A}}} \right|.
$$
Note that the polynomial $p_0+p_1p_2$ on $\bF_2^{n-2}$ is of degree at most $2$, so by definition, 
\begin{equation*}
	I_0\le \frac{1}{2}W(2,k,n-2).
\end{equation*} 
Similarly, we have 
\begin{equation*}
	I_1, I_2\le \frac{1}{2}W(2,k-1,n-2),\quad 	I_{12}\le \frac{1}{2}W(2,k-2,n-2).
\end{equation*} 
All combined, we have shown that 
\begin{equation*}
	W(2,k,n)\le \frac{1}{2}W(2,k,n-2)+W(2,k-1,n-2)+\frac{1}{2}W(2,k-2,n-2).
\end{equation*}
%
% in the expansion of $W(2,k,n)$ yields that $W(2,k,n)$  equals
%\[
%    \frac{1}{2} \frac{1}{2^{n-1}} \sum_{\bar x \in \bF_2^{n-2}} (-1)^{p_0(\bar x) + p_1(\bar x)p_2(\bar x)} \sum_{A \subset [n] \setminus \{1,2\}, |A| = k} \varepsilon_A(-1)^{x_A}
%\]
%\[
%    + \frac{1}{2} \frac{1}{2^{n-1}} \sum_{\bar x \in \bF_2^{n-2}} (-1)^{p_0(\bar x) + p_1(\bar x)p_2(\bar x) + p_2(\bar x)} \sum_{A \subset [n] \setminus \{1,2\}, |A| = k-1}\varepsilon_{A \cup \{1\}} (-1)^{x_A}
%\]
%\[
%    + \frac{1}{2} \frac{1}{2^{n-1}} \sum_{\bar x \in \bF_2^{n-2}} (-1)^{p_0(\bar x) + p_1(\bar x)p_2(\bar x) + p_1(\bar x)} \sum_{A \subset [n] \setminus \{1,2\}, |A| = k-1} \varepsilon_{A \cup {2}} (-1)^{x_A}
%\]
%\[
%    + \frac{1}{2} \frac{1}{2^{n-1}} \sum_{\bar x \in \bF_2^{n-2}} -(-1)^{p_0(\bar x) + p_1(\bar x)p_2(\bar x) + p_1(\bar x) + p_2(\bar x)} \sum_{A \subset [n] \setminus \{1,2\}, |A| = k-2} \varepsilon_{A \cup \{1,2\}} (-1)^{x_A}\,.
%\]
%The expression in the first line, for fixed $A$, is the $A$-th Fourier coefficient of the function $(-1)^{q_1(\bar x)}$, where 
%$$
%    q_1(\bar x) = p_0(\bar x) + p_1(\bar x)p_2(\bar x).
%$$
%This is a degree $2$ polynomial, because $p_1$ and $p_2$ are by assumption affine polynomials. Furthermore, it only depends on $n-2$ variables. Hence, the first line is at most $\frac{1}{2}W(2,k,n-2)$.
%Similarly, the second and third line are at most $\frac{1}{2}W(2,k-1,n-2)$, and the last line is at most $\frac{1}{2}W(2,k-2,n-2)$. It follows that
Therefore, 
\begin{align*}
    W(2, k, n) &\le \frac{1}{2} W(2, k, n-2) + W(2, k -1) + \frac{1}{2} W(2, k - 2)\\
    &\le \frac{1}{2} (1 + \sqrt{2})^{k} + (1 + \sqrt{2})^{k-1} + \frac{1}{2}(1 + \sqrt{2})^{k-2}\\
    &= (1 + \sqrt{2})^{k}\,.
\end{align*}
Here we used the induction hypotheses \eqref{induction1} and \eqref{induction2}. This completes the proof.

\section{Another approach via the structure theorem of Dickson}
\label{dickson}

Pooya Hatami kindly explained to us what proof of the estimate of  Fourier weight-level-$k$ for quadratic polynomials they had in mind in \cite{CHHL}.  
As \cite{CHHL} does not have the proof of this kind of estimate we include it here  as it seems interesting to compare it with the proof in Section \ref{firstproof}.

\medskip 
It is based on a beautiful structure theorem of Dickson (see Theorem 4 Chapter 15 of \cite{McWS}), which we formulate now in a convenient form.
First notice that
any quadratic polynomial $q:\bF_2^n \to \bF_2$ can be written in the form
$$
q(x) = x^t A x + (\ell, x) + c,
$$
where $A$ is an upper triangular matrix with $\bF_2$ entries,  $\ell$ is a binary vector and $c\in \bF_2$.

\begin{theorem}[\cite{McWS} and Lemma 2.4 \cite{BGGM}]
\label{dickson-th}
Let $q$ be as above and let $A$ be an upper triangular  binary matrix of size $n\times n$. Let $m$ be the rank of $A+A^t$. 
Then there exists a non-singular linear mapping $T: \bF_2^n\to \bF_2^n$ such that
\begin{equation}
\label{qT}
q\circ T^{-1}(y)  =\sum_{j=1}^m y_{2j-1}y_{2j} +(L, y) + b,
\end{equation}
where $L\in \bF_2^n, b\in \bF_2$.
Moreover, if $(L, y)$ is linearly dependent on $\{y_1,\dots, y_{2m}\}$, there exists another affine change of variable $\bF_2^n \to \bF_2^n$ that transfers $q\circ T^{-1}(y)$ into
$Q(z), z=(z_1,\dots, z_n)$, where
\begin{equation}
\label{Q} 
Q(z) =\sum_{j=1}^m z_{2j-1}z_{2j} + d,
\end{equation}
where $d\in \bF_2$.
 \end{theorem}

 Those changes of variables do not change the expectation over $\bF_2^n$, which we will denote by $\bE$. One simple consequence then is  the fact about the Fourier coefficients of $f:=(-1)^{q(x)}$:
 Namely, let $S\subset [n]$, then 
 \begin{align}
 	| \hat f(S)| &=|\bE (-1)^{q(x)+\sum_{k\in S} x_k}| \nonumber \\
&= \begin{cases}
 2^{-m},  &{\sum_{k\in S} x_k} \in -(L, T(x)) + \text{span} \{(T(x))_j\}_{j=1}^{2m},
 \\
 0,  \quad &\text{otherwise}\,.
 \end{cases} \label{Fc}
 \end{align}
 %\begin{equation}
 %\label{Fc}
%\!\!\!\!| \hat f(S)|\!\! = \!\!|\bE (-1)^{q(x)+\sum_{k\in S} x_k}| \!\!
%= \!\!\begin{cases}
%\!\! 2^{-m},  {\sum_{k\in S} x_k} \in\!\! -(L, T(x)) + \text{span} \{(T(x))_j\}_{j=1}^{2m},
% \\
% 0,  \quad \text{otherwise}\,.
% \end{cases}
% \end{equation}
 
 Let us estimate the number of $S$, $|S|=k$, such that the first line holds.
 
 \medskip

\subsection{Vectors of boolean weight $k$ in affine subspace of $\bF_2^n$.}
We will prove that any affine subspace of $\bF_2^n$ of dimension $2m$ contains at most ${2m+1 \choose k}$ vectors of boolean weight $k$.

%Below we will show an estimate that is only slightly worse.

\medskip

% First, the affine subspace of dimension $2m$ is contained in a subspace of dimension $2m+1$.
%We use Gaussian elimination to pick a basis $v_1, \dotsc, v_{2m+1}$ of this subspace with the following property. If $j_i, 1 \le i \le 2m+1$ is the first index $j$ such that $(v_{i})_j \ne 0$, then we have: a) $j_i$ is strictly increasing in $i$ and b) for all $i' \ne i$ it holds that $(v_{i'})_{j_i} = 0$. That such a basis exists is a consequence of the standard Gaussian elimination algorithm.

%Now any sum of more than $k$ vectors $v_j$ has weight more than $k$, because it is $1$ in particular in all of the positions $i_j$. Thus, there are at most 
%$$
%	\sum_{i = 1}^k {2m+1 \choose i} \le k {2m+1 \choose k}
%$$
%vectors of weight $k$ in the affine subspace. %We learned of this argument from Fedor Petrov's answer in \cite{P}.

\begin{lemma}
\label{lem weight k}
	Let $W \subset \bF_2^n$ be an affine subspace of dimension $h$. Then $W$ contains at most ${h+1 \choose k}$ vectors of weight $k$.
\end{lemma}

\begin{proof}
	We write $W = w + V$ for a subspace $V$ of dimension $h$.	By permuting the coordinates and applying Gaussian elimination, we can choose a basis $v ^1, \dotsc, v^h$ of $V$ that looks as follows:
	$$
		\begin{pmatrix}
			-- v^1 --\\
			-- v^2 --\\
			\vdots\\
			-- v^h --
		\end{pmatrix}
		=
		\begin{pmatrix}
		1 &0  & \hdots & 0 & * & * & \dotsb & *\\
		0 & 1  & \hdots & 0 & * &* & \dotsb & *\\
		\vdots & \vdots & \ddots & \vdots & \vdots & \vdots & & \vdots \\
		0 & 0 & \hdots & 1 & * & * & \dotsb & *
		\end{pmatrix}\,.
	$$
	Now we can further assume that $w_1 = \dotsb = w_h = 0$.
	We will prove the lemma by induction.
	
	\textbf{The lemma holds when $k  =1$.} Indeed, any sum of at least $2$ vectors $v^j$ is nonzero in at least two of the coordinates $1, \dotsc, h$, and hence its sum with $w$ does not have weight $1$. So there are at most $h+1$ vectors in $w + V$ that could have weight $1$, namely
	$$
		w, w + v^1, \dotsc, w + v^h\,.
	$$
	
	\textbf{Now let us assume that it holds for $k - 1$.} If the coefficient of $v^1$ in a weight $k$-vector $v \in w + V$ is nonzero, then we can write 
	$$
		v = \delta^1 + (w + v^1 - \delta^1) + r
	$$
	where $r \in V' := \text{span} \{ v^2, \dotsc v^h\}$ and $\delta^1$ is the vector that is one in coordinate $1$ and zero everyhwere else. Then $(w + v^1 - \delta^1) + r$ is a weight $(k-1)$ vector in an $(h-1)$-dimensional affine space. Moreover, $w' = (w + v^1 - \delta^1)$ satisfies $w'_i = 0$ for $i = 1, \dotsc, h$. By induction hypothesis, there are thus at most 
	$$
		{h \choose k-1}
	$$
	choices for $v$ with nonzero coefficient of $v^1$.
	
	On the other hand, if the coefficient of $v^1$ vanishes, then $v$ itself is a weight $k$ vector in the dimension $h-1$ affine subspace $w + V' = w+ \langle v^2,\dotsc, v^h\rangle$. By induction again, there are at most 
	$$
		{h \choose k}
	$$
	such vectors $v$. Thus, in total there are at most 
	$$
		{h \choose k-1} + {h \choose k} = {h + 1 \choose k}
	$$
	vectors $v$. This completes the proof. 
\end{proof}

The estimate in Lemma \ref{lem weight k} is sharp: If  $W$ is the affine subspace in $\bF_2^n$ generated by
$$
	 \{x \in \bF_2^n \, : \, \sum_{i=1}^{n} x_i = k\}
$$
then the dimension of $W$ is $n-1$,  and $W$ contains exactly 
$$
	{n \choose k} = { \dim W + 1 \choose k}
$$
vectors of weight $k$, namely all weight $k$ vectors in $\bF_2^n$.

\subsection{Estimating the Fourier weight}
 
Combining \eqref{Fc} and Lemma \ref{lem weight k} shows that
$$
	\sum_{|S| = k} |\hat f(S)| \le 2^{-m}{2m + 1 \choose k}\,.
$$
  
To further estimate this we start with:
  
\begin{lemma}
\label{lem alpha}
For all $\alpha \in [0,1]$, we have 
$$
	\alpha^\alpha (2- \alpha)^{2- \alpha} \ge 2 (\sqrt{2}-1)^\alpha\,.
$$
\end{lemma}

\begin{proof}	
Substituting $\alpha = 2(1 - y)$ turns this into the equivalent inequality
$$
	\left(\frac{y}{2^{-1/2}}\right)^{y} \left(\frac{1- y}{1 - 2^{-1/2}}\right)^{1 - y} \ge 1\,.
$$
It is readily verified that equality holds at $y = 2^{-1/2}$. Furthermore, if $h(y) = \exp(g(y))$ denotes the left hand side of the inequality, then 
$$
	\frac{\mathrm{d}}{\mathrm{d}y} h(y) = \exp(g(y)) g'(y) = \exp(g(y)) \left( \log \frac{y}{2^{-1/2}} - \log \frac{1-y}{1- 2^{-1/2}}\right)\,.
$$
Thus, the derivative vanishes at $y = 2^{-1/2}$. Finally, we have 
$$
	\frac{\mathrm{d}^2}{\mathrm{d}y^2} h(y) = \exp(g(y)) ((g'(y))^2 + g''(y)) \ge \exp(g(y)) g''(y)
$$
$$
	= \exp(g(y))\left( \frac{1}{y} + \frac{1}{1-y}\right)\,,
$$
so $h$ is convex and the inequality follows. 
\end{proof}

\begin{lemma}
\label{lem binomial}
For all $m \ge 1$ and $0 \le k \le m$
$$
	{2m + 1 \choose k} \le \sqrt{\frac{2}{\pi}} 2^{m} k^{-1/2} (1 + \sqrt{2})^k\,.
$$
This is sharp in the sense that there exist a sequence of $(m, k)$ such that the ratio of left hand side and right hand side tends to $1$. 
\end{lemma}

\begin{proof}
We use the estimate (see e.g. \cite{McWS}, page 309)
\begin{equation}
	\label{eq binomial}
	{n \choose k} \le \sqrt{\frac{n}{2\pi k (n-k)}} 2^{n H(k/n)}
\end{equation}
where $H$ is the binary entropy function
\[
	H(p) = -p \log_2 p - (1-p) \log_2(1-p).
\]
Note that by Stirling's formula for $k/n$ fixed and $n \to \infty$ the bound \eqref{eq binomial} is sharp, in the sense that the ratio of the left hand side and the right hand side approaches one.

Using \eqref{eq binomial}, we obtain for $k = \alpha (m + 1/2)$:
\begin{align}
	2^{-m}{2m + 1 \choose k} &\le \frac{1}{\sqrt{\pi k (2 - \alpha)}} 2^{(2m+1) H(\alpha/2)-m}\nonumber\\
	&=  \sqrt{\frac{2}{\pi k (2-\alpha)}} \left(\frac{2}{\alpha^\alpha (2- \alpha)^{(2-\alpha)}}\right)^{m+1/2}\,. \label{eq binfinal}
\end{align}
Bounding the second factor with Lemma \ref{lem alpha} and the first factor with $\alpha < 1$
$$
	\le \sqrt{\frac{2}{\pi}} k^{-1/2} (1 + \sqrt{2})^{k}\,.
$$

For sharpness we use the following claim: There exist infinitely many pairs of integers $(k,m)$ such that 
\begin{equation}
    \label{e:claim}
    \big|\frac{k}{2m + 1} - \frac{2-\sqrt{2}}{2}\big| \le \frac{1}{m^2} \qquad \iff \qquad |\alpha - (2 - \sqrt{2})| \le \frac{2}{m^2}.
\end{equation}
Using this, the fact that \eqref{eq binomial} and therefore \eqref{eq binfinal} are asymptotically sharp, and that equality in Lemma \ref{lem alpha} holds at $2 - \sqrt{2}$ yields the claimed sharpness.  

The claim \eqref{e:claim} holds since for any convergent $p/q$ in the continued fraction expansion of $1 - \sqrt{2}/2$ we have 
\[
    \left|\frac{p}{q} - \frac{2 - \sqrt{2}}{2}\right| \le \frac{1}{q^2}.
\]
Further, two consecutive convergents $p/q$ and $p'/q'$ satisfy 
\[
    pq' - qp' = \pm 1,
\]
so that at least one of $q, q'$ is odd. See \cite[Theorem 171]{HW} and related discussion for the proof.
\end{proof}

We got the desired estimate for $k\le m$. If $m+1 \le k \le 2m+1$ we just put $\ell= 2m+1 -k <k$ and get the previous estimate with $\ell$ instead of $k$. We are done:
$$
	W(2,k,n) \le \sqrt{\frac{2}{\pi}} k^{-1/2}(1 + \sqrt{2})^k,
$$
where we used that function $x\to \frac{(1+\sqrt{2})^x}{x^{1/2}}$ increases on $[1, \infty)$.
\subsection{Sharpness}
In fact, the arguments in this section were all sharp: Let 
$$
	q(x) = \sum_{j=1}^m x_{2j} x_{2j-1}
$$
and pick $(m, k)$ an asymptotically sharp sequence for Lemma \ref{lem binomial}, so 
$$
	k \approx  (2 - \sqrt{2})m\,.
$$
We use \eqref{Fc} to obtain
\[
	\sum_{|S| = k} |\hat f(S)| = 2^{-m} {2m \choose k} = 2^{-m} {2m + 1 \choose k} \frac{2m+1-k}{2m+1}.
\]
By the choice of $k$ we are in the essential equality case of Lemma \ref{lem binomial}, which yields 
$$
	\sum_{|S| = k} |\hat f(S)| \sim  \sqrt{\frac{2}{\pi}} k^{-1/2} (1 + \sqrt{2})^k\,.
$$
Thus the bound in Theorem \ref{thm deg 2 sharp} is sharp.

\section{Degree 3 case}
Here are some thoughts on the degree 3 case. Suppose that $\deg (p)=3$ and without loss of generality
\begin{align*}
     p(x)=&x_1 x_2 x_3+x_1 p_1(\bar{x})+x_2 p_2(\bar{x})+x_3 p_3(\bar{x})\\
    &\ \ \ +x_1 x_2 p_{12}(\bar{x})+x_1 x_3 p_{13}(\bar{x})+x_2 x_3 p_{23}(\bar{x})
    +p_0(\bar{x})
\end{align*}
where $x=(x_1,x_2,x_3,\bar{x})\in \bF_2^{n}$. Then
\begin{align*}
    \sum_{A\subset [n], |A|=k}\left| \frac{1}{2^n}\sum_{x\in \bF_2^n}(-1)^{p(x)+x_A}\right|
    =I_0+I_1+I_2+I_3+I_{12}+I_{13}+I_{23}+I_{123},
\end{align*}
where 
\begin{equation}
    I_0=\sum_{A\cap [3]=\emptyset,|A|=k}\left| \frac{1}{2^n}\sum_{x\in \bF_2^n}(-1)^{p(x)+x_A}\right|
\end{equation}
and for $\emptyset\neq J\subset [3]$
\begin{equation}
    I_J=\sum_{A\cap [3]=\{J\},|A|=k}\left| \frac{1}{2^n}\sum_{x\in \bF_2^n}(-1)^{p(x)+x_A}\right|.
\end{equation}
Denoting $[4,n]:=[n]\setminus [3]$ and $\bar{A}:= A\setminus [3]$. By definition, 
\begin{align*}
     I_0&=\sum_{A\cap [3]=\emptyset,|A|=k}\left| \frac{1}{2^n}\sum_{x\in \bF_2^n}(-1)^{p(x)+x_A}\right|\\
    &=\sum_{\bar{A}\subset [4,n],|\bar{A}|=k}\left| \frac{1}{2^n}\sum_{\bar{x}\in \bF_2^{n-3}}(-1)^{\bar{x}_{\bar{A}}}\sum_{x_1,x_2,x_3\in \bF_2}(-1)^{p(x)}\right|.
\end{align*}
By definition, 
\begin{align*}
&p(0,0,0,\bar{x})=p_0(\bar{x}),\\
&p(1,0,0,\bar{x})=p_1(\bar{x})+p_0(\bar{x})\\
&p(0,1,0,\bar{x})=p_2(\bar{x})+p_0(\bar{x})\\
&p(0,0,1,\bar{x})=p_3(\bar{x})+p_0(\bar{x})\\
&p(1,1,0,\bar{x})=p_1(\bar{x})+p_2(\bar{x})+p_{12}(\bar{x})+p_0(\bar{x})\\
&p(1,0,1,\bar{x})=p_1(\bar{x})+p_3(\bar{x})+p_{13}(\bar{x})+p_0(\bar{x})\\
&p(0,1,1,\bar{x})=p_2(\bar{x})+p_3(\bar{x})+p_{23}(\bar{x})+p_0(\bar{x})\\
&p(1,1,1,\bar{x})=1+p_1(\bar{x})+p_2(\bar{x})+p_3(\bar{x})+p_{12}(\bar{x})+p_{13}(\bar{x})+p_{23}(\bar{x})+p_0(\bar{x}).
\end{align*}
Thus, 
\begin{align}
\label{p00}
    \sum_{x_1,x_2,x_3\in \bF_2}(-1)^{p(x)}
    &=(-1)^{p_0(\bar{x})}                                      
    \left[
    1+(-1)^{p_1(\bar{x})}+(-1)^{p_2(\bar{x})}+(-1)^{p_3(\bar{x})}\right.           \notag
    \\ 
    &+\left.(-1)^{p_1(\bar{x})+p_2(\bar{x})+p_{12}(\bar{x})}+(-1)^{p_1(\bar{x})+p_3(\bar{x})+p_{13}(\bar{x})}+(-1)^{p_2(\bar{x})+p_3(\bar{x})+p_{23}(\bar{x})}\right.           \notag
    \\           
        &+\left. (-1)^{1+p_1(\bar{x})+p_2(\bar{x})+p_3(\bar{x})+p_{12}(\bar{x})+p_{13}(\bar{x})+p_{23}(\bar{x})}
    \right].
\end{align}
We have the following lemma. 

\begin{lemma}
    We have the following identity: for any $a,b,c,x,y,z\in \bF_2$
    \begin{align*}
        & 1+(-1)^{a}+(-1)^{b}+(-1)^{c}+(-1)^{a+b+z}+
        (-1)^{a+y+c}+
        (-1)^{x+b+c}\\
        &+(-1)^{1+a+b+c+x+y+z}\\
=&2(-1)^{ab}+2(-1)^{ab+ax+by+xy+c}+2(-1)^{a+b+xz+yz+cz}.
    \end{align*}
\end{lemma}

\begin{proof}
    Recall that we have for all $\alpha,\beta\in \bF_2$
    \begin{equation}
        1+(-1)^{\alpha}+(-1)^{\beta}+(-1)^{1+\alpha+\beta}=2(-1)^{\alpha\beta}.
    \end{equation}
    Then we have 
    \begin{align}
        1+(-1)^{a}+(-1)^{b}=2(-1)^{ab}-(-1)^{1+a+b},
    \end{align}
    \begin{align*}
         (-1)^{c}+(-1)^{a+y+c}+(-1)^{x+b+c}
         &=(-1)^c\left[1+(-1)^{a+y}+(-1)^{x+b}\right]\\
         &=(-1)^c\left[2(-1)^{(a+y)(x+b)}
         -(-1)^{1+a+b+x+y}\right]\\
         &=2(-1)^{c+ab+ax+by+xy}-(-1)^{1+a+b+c+x+y},
    \end{align*}
and 
\begin{align*}
(-1)^{a+b+z}+(-1)^{1+a+b+c+x+y+z}
&=(-1)^{a+b}[(-1)^{z}+(-1)^{1+c+x+y+z}]\\
&=(-1)^{a+b}\left[(-1)^{z}+(-1)^{1+c+x+y+z}\right]\\
&=(-1)^{a+b}\left[
2(-1)^{(c+x+y)z}-(-1)^{c+x+y}-1\right]\\
&=
2(-1)^{a+b+cz+xz+yz}-(-1)^{a+b+c+x+y}-(-1)^{a+b}.
\end{align*}
Summing all these identities finishes the proof.
\end{proof}

According to this lemma, we may rewrite 
\begin{align}
\label{p0}
    \sum_{x_1,x_2,x_3\in \bF_2}(-1)^{p}
    =&2(-1)^{p_0+p_1 p_2}       \notag
    +2(-1)^{p_0+p_1 p_2
    +p_1 p_{23}+p_2 p_{13}+p_{13}p_{23}+p_3}\\
    &+2(-1)^{p_0+p_1 +p_2+p_{12}p_{23} +p_{12}p_{13}+p_3 p_{12}}.
\end{align}
This allows us to reduce eight terms into six terms. Recall that 
$\deg(p_i)\le 2$ and $\deg(p_{ij})\le 1$, so we see that most new polynomials have degrees at most $3$, except for $p_1 p_2$. Can we solve this problem by rewriting the identity in the above lemma? 

The above discussion for $I_0$ applies to $I_J,J\neq\emptyset$ in a similar way. Take $I_1$ for example. By definition, 
\begin{align*}
     I_1&=\sum_{A\cap [3]=\{1\},|A|=k}\left| \frac{1}{2^n}\sum_{x\in \bF_2^n}(-1)^{p(x)+x_A}\right|\\
    &=\sum_{\bar{A}\subset [4,n],|\bar{A}|=k-1}\left| \frac{1}{2^n}\sum_{\bar{x}\in \bF_2^{n-3}}(-1)^{\bar{x}_{\bar{A}}}\sum_{x_1,x_2,x_3\in \bF_2}(-1)^{p(x)+x_1}\right|.
\end{align*}
So, to rewrite the sum of powers of $-1$, we only need to replace $p$ in $I_0$ with $p+x_1$. Recalling the definition of $p$, we only need to replace $p_1$ with $p_1+1$. That is, 
\begin{align*}
    \sum_{x_1,x_2,x_3\in \bF_2}(-1)^{p+x_1}
    =&2(-1)^{(p_1+1) p_2}
    +2(-1)^{(p_1+1)p_2
    +(p_1+1) p_{23}+p_2 p_{13}+p_{13}p_{23}+p_3}\\
    &+2(-1)^{(p_1+1) +p_2+p_{12}p_{23} +p_{12}p_{13}+p_3 p_{12}}.
\end{align*}
Again, the main issue is the degree of $p_1 p_2$. 

In general, for $I_J$ with $\emptyset\neq J\subset [3]$, the difference is that one replaces all $p_i,i\in J$ with $p_i+1$.

\section{Other arguments}
\label{other}

The other argument to deal with degree $3$ by using the identity 
$$
    (-1)^\alpha + (-1)^\beta + (-1)^\gamma - (-1)^{\alpha + \beta + \gamma} = 2 (-1)^{\alpha \beta + \alpha \gamma + \beta \gamma}
$$
leads also to higher degree polynomial terms, and we have no similar identity for degree $4$, so we cannot iterate there. 

However,  one can prove the following theorems.
\begin{theorem}
\label{t-p0}
Suppose that polynomial $p$ is such that $\deg p=3$ but $\deg p_0=2$.
Then the $k$-th Fourier weight of $(-1)^p$ is at most $C(1+\sqrt{2})^k$ for all $k$.
\end{theorem}

This theorem can be used by tracking the consequences of \eqref{p00} if $\deg p_0\le 2$.  In fact, the assumption says that
$$
p= x_1 \cdot p_1 + x_2\cdot p_2 + x_3\cdot p_3 + p_0,
$$
where all $p_i, i=0,1,2,3$  are degree $2$ polynomials. The same simple tracking of \eqref{p00}  will work for any 
constant $L$ independent of $n$ if $p$ is of degree $2$ in $x_{L+1}, ... , x_n$.
So this theorem is a very modest improvement. 

The next result is also based on the  decomposition of degree $3$  polynomial $p$ into a restricted number of products of linear and quadratic polynomials  plus (very roughly speaking)
another polynomial of degree $3$ of much smaller number of variables.  The next result uses the definition of bias of $p$, see e.g. \cite{HS}.

\medskip

Bias of $p: \bF_2^n \to \bF_2$ is
$$
\text{bias}(p) =\big| \bE (-1)^p\big|\,.
$$

Not too close to zero bias means that the distribution created by $p$ is not too close to the uniform distribution.
Random polynomials have small bias. So the assumption is that $p$ is far from being random. Kaufman and Lovett \cite{KL}  (see also \cite{GT} for a close but different situation) proved that polynomial on $\bF_2^n$
with very large $n$ and of degree $d$ with noticeable bias can be represented as a certain function of  smaller degrees whose number depends on $c=c(d, \text{bias})$.

\begin{theorem}
\label{bias}
Suppose that polynomial $p$ is such that $\deg p=3$ but $\text{bias} (p) \ge \delta$.
Then the $k$-th Fourier weight of $(-1)^p$ is at most $(C/\delta)^k$ for all $k$.
\end{theorem}

The proof of this theorem in \cite{HS} is based on a structure theorem (not unlike Dickson's lemma used above) that claims that a polynomial of degree $3$ with a noticeable bias $\delta$ 
has a decomposition to the sum of  products of linear and quadratic polynomials, where the sum
has at most $c\log\frac1\delta$ terms plus a cubic polynomial of certain amount of 
linear terms, whose number is again controlled by $\delta$ (and independent of $n$).

\section{Why the induction argument for $k > 2$ does not work}
\label{}
We use three things:
\begin{equation}
\label{d1}
W(d, 1, n) \le 4d \le 2^{3d}.
\end{equation}
\begin{equation}
\label{2k}
W(2, k, n) \le C_0^k 2^{3\cdot 2}\le (C_0 2^{3\cdot d} +1)^k, \quad d=2\,.
\end{equation}
and Lemma 6.2 of \cite{CHHL}:
\begin{lemma}[CHHL]\label{sq}
$$
W(d, k, n)^2\le 2^{2k} W(d-1, 2k, n)  + W(d, k, n)\cdot \sum_{\ell=1}^k {k\choose \ell} W(d, k-\ell, n)\,.
$$
\end{lemma}

We would like to use the above lemma, along with\eqref{d1} and \eqref{2k} to prove by induction on $d$ and $k$ the estimate
\begin{equation}
W(d,k, n) \le (C_0 \cdot 2^{3\cdot  d} +1)^k\,.
\end{equation}
By \eqref{d1} and \eqref{2k} the estimate holds for $d = 1, 2$ and all $k$. Our goal is to show it for the pair $(d,k)$, and we assume by induction that it holds for all pairs 
$(d', k')$ such that $d < d'$ or such that $d = d'$ and $k' < k$. By Lemma \ref{sq} and the induction hypothesis, we have 
\begin{align*}
    W(d,k, n)^2 &\le 2^{2k} W(d-1,2k, n) + W(d,k, n) \cdot \sum_{l = 1}^k {k \choose l} W(d, k-l, n)\\
    &\le 2^{2k}(C_0\cdot 2^{3\cdot (d-1)} + 1)^{2k} + W(d,k, n) \cdot \sum_{l = 1}^k {k \choose l} (C_0 2^{3d} + 1)^{k-l}\\
    &= 2^{2k}(C_0\cdot 2^{3\cdot (d-1)} + 1)^{2k} + W(d,k, n) ( (C_0\cdot  2^{3\cdot d} + 2)^k - (C_0\cdot 2^{3\cdot d} + 1)^k)\,.
\end{align*}
Again, we would like to show that $W(d,k) \le (C_0 \cdot 2^{3 \dot d} + 1)^k$. Suppose that it were bigger,  i.e.
$$
W(d,k, n) > (C_0 \cdot 2^{3 \dot d} + 1)^k,
$$
then we can divide by $W(d,k)$ and obtain 
\begin{align}
    \label{upperboundW}
    W(d,k, n) \le \left(\frac{2 (C_0 2^{3(d-1)} + 1)}{C_0 2^{3d} + 1}\right)^k + ( (C_0  2^{3d} + 2)^k - (C_0 2^{3d} + 1)^k)
\end{align}
We have 
$$
    \frac{2 (C_0 2^{3(d-1)} + 1)}{C_0 2^{3d} + 1} = \frac{\frac{1}{4} C_0 2^{3d} + 2}{C_0 2^{3d} + 1} \le \frac{C_0 2^{3d} + 1}{C_0 2^{3d} + 1} = 1\,.
$$
We want to show that the second term in \eqref{upperboundW} becomes much larger than $(C_0 \cdot 2^{3\cdot d} + 1)^k$ when $k$ is very large. We have for general $A$ that 
$$
    (A + 1)^k - A^k = A^k((1 + \frac{1}{A})^k - 1) \ge A^k\frac{k}{A}\,.
$$
Using this for $A = C_0 \cdot 2^{3\cdot d} + 1$, we see that when $k > C_0\cdot 2^{3\cdot d} + 1$, then 
$$
    (C_0  2^{3d} + 2)^k - (C_0 2^{3d} + 1)^k > (C_0 2^{3d} + 1)^k\,,
$$
so we do not get the contradiction we would like to have. 

\bigskip

However,
for $k\le 3d$ we have ($A = C_0 \cdot 2^{3\cdot d} + 1$)
$$
1+ (A + 1)^k - A^k =1+ A^k((1 + \frac{1}{A})^k - 1) \le 1+ A^k\frac{2^k}{A} \le A^k\,.
$$
The estimate works as long as $k \ll 2^d$, but this is not helpful, because in the induction argument we pass to much larger values of $k$ than what we started with. (In fact, we have to use the lemma $d - 2$ times to get to the bound for $d = 2$, so the largest $k$ we get is exactly $2^{d-2}$ times the $k$ we started with.

However,
this is what we get as the estimate of $W(d, k, n)$ for a certain range of $d, k$:
\begin{equation}
\label{certain}
W(d, k, n) \le (C_0 2^{3d} +1)^k  \le (2C_0 2^{3d})^k\quad \,\,\text{if}\,\,{k\le 3d}\,.
\end{equation}

\begin{comment}
\section{Pseudo random generators (PRG)}
\label{prg}

In \cite{CHLT} it is proved that if the class $\mathcal{F}$ of  Boolean functions  is closed under restriction and uniformly
$$
\mathcal{L}_{1, k}(\mathcal{F}) \le a\cdot b^k,\quad k=1, \dots, n,
$$
then for any $\eps>0$ there exists an explicit PRG for $\mathcal{F}$ with error $\eps$ and seed length $s= b^2\cdot poly\,log (an/\eps)$.

\medskip

Put $\mathcal{F}:=\{\1^p, \deg p\le 2\}$.	By Theorem \ref{thm deg 2} this class of functions satisfy the above condition. So, we have the following

\begin{theorem}
\label{prgth}
Let $\mathcal{F}:=\{\1^p, \deg p\le 2\}$.
For any $\eps>0$ there exists an explicit PRG for this $\mathcal{F}$ with error $\eps$ and seed length $s= (1+\sqrt{2})^2\cdot poly\,log (\frac{8n}{(1+\sqrt{2})\eps})$.
\end{theorem}
\end{comment}


\begin{thebibliography}{999}

\bibitem[BGGM19]{BGGM} {\sc A. Bhattacharyya,  Ph. George John, S. Ghoshal, R. Meka}, {\em Average Bias and Polynomial Sources},  Electronic Colloquium on Computational Complexity, Report No. 79 (2019), pp. 1--20.
		
\bibitem[CHLT19]{CHLT}	{\sc E. Chattopadhyay, 	P.  Hatami,  S. Lovett, A. Tal}, {\em Pseudorandom generators from the second Fourier level and applications to $AC0$ with parity gates},
10th Innovations in Theoretical Computer Science (ITCS 2019). Editor: Avrim Blum; Article No. 22; pp. 22:1--22:15.
Leibniz International Proceedings in Informatics
Schloss Dagstuhl – Leibniz-Zentrum für Informatik, Dagstuhl Publishing, Germany.

\bibitem[CGLLS20]{CGLLS} {\sc E. Chattopadhyay, , J. Gaitonde, C. H. Lee, S. Lovett, A. Shetty}, {\em Fractional pseudorandom generators from any Fourier level},
Electronic Colloquium on Computational Complexity, Revision 2 of Report No. 121 (2020), PP. 1--24.

\bibitem[CHHL18]{CHHL} {\sc E. Chattopadhyay, P. Hatami, K. Hosseini, S. Lovett}, {\em Pseudorandom Generators from Polarizing Random Walks},
33rd Computational Complexity Conference,
CCC 2018, pp. 1:1–1:21, 2018.

\bibitem[Gol10]{G} {\sc O. Goldreich},  {\em A Primer on Pseudorandom Generators}. Volume 55 of Univ. Lecture Ser.
Amer. Math. Soc., 2010.

\bibitem[GT08]{GT}{\sc B. Green, T. Tao}, {\em The distribution of polynomials over finite fields, with
applications to the Gowers norms},  arXiv:0711.3191, 2007, Contributions to Discrete Mathematics, v. 4, No. 2, (2008) pp.1--36

%\bibitem[P]{P} {\sc F. Petrov}, Low-Hamming weight vectors in low-dimensional subspaces of $\mathbb{F}_p^n$, URL: https://mathoverflow.net/q/389026

\bibitem[HS09]{HS} {\sc E. Haramaty, A. Shpilka}, {\em On the Structure of Cubic and Quartic Polynomials}, 	arXiv:0908.2853, pp. 1--31.

\bibitem[HW08]{HW}{\sc
G.~H. Hardy and E.~M. Wright}, {\it An introduction to the theory of numbers}, sixth edition, 
Oxford Univ. Press, Oxford, 2008.

\bibitem[HLW06]{HLW} {\sc S. Hoory, N. Linial,  A. Widgerson}, {\em Expander graphs and their applications}.
Bull. Amer. Math. Soc., 43(4) (2006), 439--561.

\bibitem[KL08]{KL} {\sc T. Kaufman and S. Lovett}, {\em  Worst case to average case reductions for polynomials.}
In 49th Annual FOCS, pages 166--175, 2008.

\bibitem[McWS77]{McWS} {\sc F. J. McWilliams, N. J. A. Sloane}, {\em The Theory of Error Correction Codes},  North Holland mathematical Library, v. 16, 1977.

\bibitem[NW94]{NW} {\sc N. Nisan, A. Widgerson}, {\em Hardness vs randomness}. J. Comput. System Sci., 49(2), (1994) pp. 149--
167.



	\end{thebibliography}
\end{document}